\documentclass[11pt]{amsart}
\usepackage{amsmath}
\usepackage{amssymb}
\usepackage{amsfonts}
\usepackage{latexsym}
\usepackage{color}
\usepackage{graphicx}
\usepackage{appendix}
\usepackage{bm}

\newcommand{\be}{\begin{equation}}
\newcommand{\en}{\end{equation}}
\newcommand{\bea}{\begin{eqnarray}}
\newcommand{\ena}{\end{eqnarray}}
\newcommand{\beano}{\begin{eqnarray*}}
\newcommand{\enano}{\end{eqnarray*}}
\newcommand{\bee}{\begin{enumerate}}
\newcommand{\ene}{\end{enumerate}}

\newcommand{\mc}{\mathcal}
\newcommand{\mb}{\mathbb}

\newcommand{\F}{{\mathcal F}}

\newcommand{\D}{{\mc D}}

\newcommand{\E}{{\mc E}}

\newcommand{\M}{{\sf M}}

\newtheorem{defn}{Definition}[section]
\newtheorem{thm}[defn]{Theorem}
\newtheorem{prop}[defn]{Proposition}
\newtheorem{lemma}[defn]{Lemma}
\newtheorem{cor}[defn]{Corollary}
\newtheorem{example}[defn]{Example}
\newtheorem{rem}[defn]{Remark}

\def\x{\relax\ifmmode {\mbox{*}}\else*\fi}
\newcommand{\beex}{\begin{example}$\!\!${\bf }$\;$\rm }
\newcommand{\enex}{ \end{example}}
\newcommand{\berem}{\begin{rem}$\!\!${\bf }$\;$\rm }
\newcommand{\enrem}{ \end{rem}}
\newcommand{\bedefi}{\begin{defn}$\!\!${\bf }$\;$\rm }
\newcommand{\findefi}{\end{defn}}

\newcommand{\ha}{^{\ast}}

\newcommand{\ip}[2]{\left\langle {#1}\left|{#2}\right.\right\rangle}

\def\H{{\mathcal H}}

\newcommand{\NN}{{\mathbb N}}
\newcommand{\dis}{\displaystyle}
\newcommand{\sss}{\scriptscriptstyle}
\newcommand{\gs}{\Psi}
\newcommand{\ggs}{\Omega}
\newcommand{\sggs}{_{\sss\ggs}}
\newcommand{\gi}{\iota}
\newcommand{\emb}{{\sf u}}
\newcommand{\gn}{\mathfrak{n}}
\newcommand{\gb}{\Upsilon}
\newcommand{\sgb}{_{\sss\gb}}
\newcommand{\ds}{\displaystyle}

\def\ker{{\sf Ker\,}}
\def\ran{{\sf Ran\,}}
\def\Re{\mathfrak {Re}\,}
\def\Im{\mathfrak {Im}\,}
\begin{document}
\title[Some Representation Theorems]
{Some Representation Theorems for Sesquilinear Forms}

\author{Salvatore di Bella}
\author{Camillo Trapani}
\address{Dipartimento di Matematica e Informatica,
Universit\`a di Palermo, I-90123 Palermo, Italy}

\email{svt.dibella@gmail.com; salvatore.dibella@unipa.it}
\email{camillo.trapani@unipa.it}


\begin{abstract} The possibility of getting a Radon-Nikodym type theorem and a Lebesgue-like decomposition for a non necessarily positive sesquilinear $\Omega$ form defined on a vector space $\D$, with respect to a given positive form $\Theta$ defined on $\D$, is explored.
The main result consists in showing that a sesquilinear form $\Omega$ is $\Theta$-regular, in the sense that it has a  Radon-Nikodym type representation, if and only if it satisfies a sort Cauchy-Schwarz inequality whose right hand side is implemented by a positive sesquilinear form which is $\Theta$-absolutely continuous. In the particular case where $\Theta$ is an inner product in $\D$, this class of sesquilinear form covers all standard examples. In the case of a form defined on a dense subspace $\D$ of Hilbert space $\H$ we give a sufficient condition for the equality $\Omega(\xi,\eta)=\ip{T\xi}{\eta}$, with $T$ a closable operator, to hold on a dense subspace of $\H$.
\end{abstract}

\maketitle

\section{Introduction}\label{sect_introd}
It is a very basic fact that to every linear operator $T$ defined on a subspace $\D$ of a Hilbert space $\H$ there corresponds a sesquilinear form $\Omega_T$ on $\D\times \D$ defined as
$$ \Omega_T (\xi, \eta)= \ip{T\xi}{ \eta}, \quad \xi, \eta\in \D,$$ and $\Omega_T$ is named the {\em sesquilinear form associated to $T$}.
It is certainly more and more interesting to consider the converse question: given a sesquilinear form $\Omega$ on $\D\times \D$ does there exists a linear operator $T$ such that $\Omega=\Omega_T$?
The problem has very well-known solutions if the Hilbert space is finite-dimensional or if the form $\Omega$ is bounded. The situation changes dramatically for unbounded sesquilinear forms.
Although several conditions on a sesquilinear form $\Omega$ are known (see, for instance \cite{kato, reedsimon, schmu}) for $\Omega$ to be the sesquilinear form associated to a linear operator $T$ in $\H$, a complete answer to the question remains unknown, in spite of the fact that this question has been taken under consideration by several authors (for the non-semibounded case, see, for instance, \cite{mcintosh, fleige, fleigehassidesnoo} and for a treatment in Kre\u{\i}n spaces \cite{hassikrein}). The most relevant results are the first and second representation Kato theorems concerning, respectively, closed sectorial forms and positive, or semibounded, forms (semiboundedness, in particular, seemed for long time to be an ineludible condition \cite{kato, edmundevans}). Sectoriality and positivity are clearly condition on the numerical range of the form. On the other hand, the closedness condition on sectorial form $\Omega$ means, roughly speaking that $\D$ can be made into a Hilbert space under a new norm $\|\cdot\|_\Omega$, generating a topology finer than the initial one,  and with respect to which  $\Omega$ is bounded. A construction of this kind is proposed in this paper for a sesquilinear form $\Omega$ with no particular assumptions on their numerical range in order to get a Radon-Nikodym theorem type for them. This means, if $\Omega$ is a sesquilinear form on $\D\times \D$, with $\D$ a dense subspace of Hilbert space,  that we look for a representation of the following kind:
\begin{equation}\label{eqn_introd}\Omega(\xi, \eta)= \ip{HY\xi}{H\eta}, \quad \forall \xi, \eta\in \D,\end{equation}
where $H$ is a positive self-adjoint operator in $\H$  and $Y$ is so that $HY$ is well-defined and closable in $\D$.

We consider this problem in a more general setting, by considering se\-squi\-li\-near forms $\Omega$ defined on a complex vector space $\D$. Properties of $\Omega$ are referred to a fixed positive sesquilinear form $\Theta$ on $\D\times \D$  in the same spirit of Hassi, Sebesty\'en and de Snoo  in \cite{hassi}; we will not follow however their approach but we prefer to use only Operator theory methods (we borrow, in fact, some techniques used in \cite[Ch. 9]{ait_book} and in \cite{Inoue} for certain functionals or positive sesquilinear forms on (partial) *-algebras; see also \cite{larussatriolo}).

 In Section \ref{sect_radon_nikodym} we consider a sesquilinear form $\Omega$ for which the set $\M(\Omega)$ of all positive sesquilinear forms $\Psi$ on $\D \times \D$ such that
$$ |\Omega(\xi,\eta)|\leq \Psi (\xi,\xi)^{1/2} \Psi (\eta,\eta)^{1/2}, \quad \forall \xi, \eta \in \D$$ is nonempty. Then we introduce the notion of $\Theta$-regular form and show that this is equivalent to $\Omega$ admitting a representation of the type \eqref{eqn_introd} (with an extra condition concerning the $\Theta$-absolute continuity of a certain positive $\Gamma_{\sss \Omega}$ constructed from $H$ and $Y$).

 In Section \ref{sect_lebesgue} we show that every sesquilinear form $\Omega$ for which $\M(\Omega)$ is nonempty allows a Lebesgue-like decomposition into the sum of a $\Theta$-regular and a $\Theta$-singular form (both non necessarily positive), generalizing in this way a series of results mostly concerned with semibounded forms \cite{simon, hassi,titkos}. It is apparent that the condition $\M(\Omega)\neq \emptyset$ plays here the same role that bounded variation plays for the classical Lebesgue decomposition theorem on measure spaces. A Lebesgue decomposition gives a natural relevance to singular forms, which, at least in the positive case, occur frequently in applications \cite{koshma}.

In Section \ref{sect_solvable} we focus on the case when $\D$ is a pre-Hilbert space and $\Theta$ is exactly the inner product of $\D$. In this situation we study the possible representation of a sesquilinear form $\Omega$ as $\Omega_T$, for some convenient operator $T$ defined on a subspace of $\D$. For this we introduce two new notions: that of {\em q-closed} form, which means, roughly speaking that $\D$ can be made into a Banach space under a new norm $\|\cdot\|_\Omega$, which makes possible the construction of a Banach-Gelfand triplet of spaces, and $\Theta$ is $\|\cdot\|_\Omega$-bounded on the smallest space of the triplet.  The second notion we introduce for $\Omega$ is that of {\em solvability}, this roughly means that $\Omega$ can be perturbed by a bounded sesquilinear form such that the corresponding operator acting in the triplet is bounded with bounded inverse. Under this assumption we prove that $\Omega=\Omega_T$, for some closed operator $T$. If this perturbation is a scalar multiple of the inner product this condition of solvability turns out to be a condition on the numerical range $\gn\sggs$ of $\Theta$ which simply means that $\gn\sggs$ does not fill the whole complex plane.

\section{Notations and preliminaries}\label{sect_preliminaries}

\label{sect_repres_sesq}
Let $\D$ be a complex vector space and $\Psi$ a sesquilinear form on $\D \times \D$.
 As usual, the {\em adjoint} form $\gs^*$ is defined by
$$ \gs^* (\xi, \eta)= \overline{\gs(\eta, \xi)}, \quad \xi, \eta \in \D.$$ If $\gs^*= \gs$, then $\gs$ is {\em symmetric}.
We also put

$$ \Re \gs = \frac{1}{2}(\gs +\gs^*) , \qquad \Im \gs = \frac{1}{2i}(\gs -\gs^*).$$
Both $\Re \gs$ and $\Im \gs$ are symmetric sesquilinear forms on $\D\times \D$ and
$$ \gs= \Re\gs +i \Im \gs.$$

 We set
$$ {\sf N}(\Psi) =  \{\xi \in \D: \Psi(\xi,\eta)=0, \, \forall \eta \in \D\}.$$
Then ${\sf N}(\Psi)$ is a subspace of $\D$.

If $\Psi$ is positive (i.e., $\Psi(\xi, \xi)\geq 0$, for every $\xi \in \D$), then it is symmetric and
$${\sf N}(\Psi) =\{\xi \in \D: \Psi(\xi,\xi)=0\}.$$
In this case, we can consider the quotient $\D/ {\sf N}(\Psi)$. We put $j_{\sss{\Psi}} (\xi):= \xi + {\sf N}(\Psi)$. Then $j_{\sss{\Psi}}(\D)=\D/ {\sf N}(\Psi)$ is a pre-Hilbert space with inner product
$$ \ip{j_{\sss{\Psi}}(\xi)}{j_{\sss{\Psi}}(\eta)}_{\sss{\Psi}}= \Psi(\xi, \eta).$$
We denote by $\H_{\sss{\Psi}}$ the Hilbert space completion of $j_{\sss{\Psi}}(\D)$. Its inner product will be denoted by $\ip{\cdot}{\cdot}_{\sss{\Psi}}$ and the corresponding norm by $\|\cdot\|_{\sss{\Psi}}$.

If $\Theta$, $\Psi$ are positive sesquilinear forms on $\D\times \D$  we write $\Theta \leq \Psi$ if $\Theta(\xi,\xi)\leq \Psi(\xi,\xi)$, for every $\xi \in \D$.
If $\Theta$, $\Psi$  are positive, we say that $\Psi$ {\em dominates} $\Theta$ if there exists $\gamma>0$ such that $\Theta \leq \gamma \Psi$; i.e.,
\begin{equation} \label{eq_oone}\Theta(\xi,\xi) \leq \gamma \Psi(\xi,\xi), \quad \forall \xi \in \D. \end{equation}

\section{Absolutely continuous and regular sesquilinear forms}\label{sect_radon_nikodym}
In this section we will extend to non necessarily positive sesquilinear forms  the notion of absolute continuity and study the possibility of getting a Radon-Nikodym-like theorem for them. But before doing this we need some preliminaries.

\begin{lemma}\label{lemma_oone} Let $\Psi$ dominate $\Theta$. Then there exists a bounded operator $C$ on $\H_{\sss{\Psi}}$, with $0\leq C\leq \gamma I$, such that
$$ \Theta(\xi,\eta)= \ip{Cj_{\sss{\Psi}}(\xi)}{j_{\sss{\Psi}}(\eta)}_\Psi, \quad \forall \xi, \eta \in \D.$$
\end{lemma}
\begin{proof}
By \eqref{eq_oone} it follows easily that $C_o:j_{\sss{\Psi}}(\xi) \to j_{\sss{\Theta}}(\xi)$ is a well defined linear operator from $j_{\sss{\Psi}}(\D)$ into $j_{\sss{\Theta}}(\D)$ and, also, that it is bounded. Thus, it extends to a bounded operator (denoted by the same symbol) from $\H_{\sss{\Psi}}$ into $\H_{\sss{\Theta}}$. We put $C:= C_o^*C_o$. Then $C$ is a bounded operator on $\H_{\sss{\Psi}}$ with $0\leq C\leq \gamma I$.
One has,
$$ \ip{Cj_{\sss{\Psi}}(\xi)}{j_{\sss{\Psi}}(\eta)}_\Psi = \ip{C_oj_{\sss{\Psi}}(\xi)}{C_oj_{\sss{\Psi}}(\eta)}_{\sss{\Theta}}= \ip{j_{\sss{\Theta}}(\xi)}{j_{\sss{\Theta}}(\eta)}_{\sss{\Theta}}=\Theta(\xi,\eta).$$
\end{proof}

Let $\Omega$ be a sesquilinear form on $\D\times \D$. We denote by $\M(\Omega)$ the set of all positive sesquilinear forms $\Psi$ on $\D\times \D$ such that
\begin{equation}\label{eq:domination} |\Omega(\xi,\eta)|\leq \Psi (\xi,\xi)^{1/2} \Psi (\eta,\eta)^{1/2}, \quad \forall \xi, \eta \in \D.\end{equation}

{
%
\berem It is worth remarking that if \begin{equation} \label{eqn_Mom}|\Omega(\xi,\xi)|\leq \Psi(\xi,\xi), \quad \forall \xi\in \D\end{equation} then
$$|\Omega(\xi,\eta)| \leq \epsilon_\Omega \Psi(\xi,\xi)^{1/2} \Psi(\eta,\eta)^{1/2}, \quad \forall \xi,\eta \in D,$$
where either $\epsilon_\Omega=1$, if $\Omega$ is symmetric, or $\epsilon_\Omega=2$, if $\Omega$ is not symmetric (see, e.g. \cite[Ch. VI, $\S$ 1.2, $\S$ 3.1]{kato}).
Hence, $\M(\Omega)$ can also be described as the set of all positive forms $\Psi$ for which the inequality \eqref{eqn_Mom} holds.
\enrem
}
\bedefi \label{defn_abs_cont} Let  $\Theta$  be a positive sesquilinear form on $\D\times \D$. A positive sesquilinear form $\Psi$ is said to be {\em $\Theta$-absolutely continuous} if

(i) ${\sf N}(\Theta) \subseteq {\sf N}(\Psi)$;

(ii) The map $j_{\sss{\Theta}} (\xi) \to j_{\sss{\Psi}}(\xi)$, $\xi\in \D$  is a closable linear map of the pre-Hilbert space $j_{\sss{\Theta}}(\D)$ into the Hilbert space $\H_{\sss{\Psi}}$.

A (non necessarily positive) sesquilinear form $\Omega$ is said to be {\em $\Theta$-regular} if there exists $\Psi \in \M(\Omega)$ such that $\Psi$ is $\Theta$-absolutely continuous. \findefi
\berem If $\Psi$ is a positive sesquilinear form, then $\Psi \in \M(\Psi)$ and if $\Psi$ is $\Theta$-absolutely continuous, then it is $\Theta$-regular. We point out that (i) can be deduced from (ii), but since it constitutes a preliminary test for the $\Theta$-absolute continuity of a positive form $\Psi$, we prefer to keep them separate.\enrem

{It is easily seen that the sum of $\Theta$-regular forms is $\Theta$-regular and so is the positive scalar multiple of a $\Theta$-regular form.}

\beex \label{example_35}
Let $T$ be a closed operator defined on a dense domain $D(T)$ of a Hilbert space $\H$. Assume that $\D:=D(T)\cap D(T^*)$ is dense in $\H$. We define a sesquilinear form $\Omega_T$ on $\D \times \D$ by putting
$$ \Omega_T(\xi, \eta)= \ip{T\xi}{\eta}, \; \xi, \eta \in \D$$
Let $T=UH$ be the polar decomposition of $T$ with $H=(T^*T)^{1/2}$.
Let us define a positive sesquilinear form $\Gamma_T$ on $\D\times\D$ by
$$\Gamma_T(\xi, \eta)=\ip{\xi}{\eta}+\ip{H\xi}{\eta}+ \ip{HU^*\xi}{U^*\eta}, \quad \xi, \eta \in \D_T.$$
Then $\Gamma_T \in \M(\Omega_T)$. Indeed, by the generalized Cauchy-Schwarz inequality,  we have, for every $\xi, \eta \in \D$,
\begin{align*}
|\Omega_T(\xi, \eta)|&= |\ip{T\xi}{\eta}|= |\ip{UH\xi}{\eta}|= |\ip{H\xi}{U^*\eta}|\\
&\leq \ip{H\xi}{\xi}^{1/2} \ip{HU^*\eta}{U^*\eta}^{1/2} \leq \Gamma_T(\xi, \xi)^{1/2} \Gamma_T(\eta, \eta)^{1/2}  .\end{align*}
Now we show that $\Gamma_T$ is $\Theta$-absolutely continuous where $\Theta$ is the inner product of $\H$. Indeed, let $\{\xi_n\}$ be a sequence in $\D$ such that $\|\xi_n\|\to 0$ and $\Gamma_T(\xi_n- \xi_m,\xi_n- \xi_m)\to 0$.
This implies that
\begin{align*}
\Gamma_T(\xi_n- &\xi_m,\xi_n- \xi_m) = \|\xi_n - \xi_m\|^2+\ip{H(\xi_n- \xi_m)}{\xi_n- \xi_m}\\&+ \ip{HU^*(\xi_n- \xi_m)}{U^*(\xi_n- \xi_m)}\\
&= \|\xi_n - \xi_m\|^2+\|{H^{1/2}(\xi_n- \xi_m)}\|^2+ \|{H^{1/2}U^*(\xi_n- \xi_m)}\|^2\to 0.
\end{align*}
Since $H^{1/2}$ is closable in $\D_T$, we get  $\Gamma_T(\xi_n, \xi_n)\to 0$.
Hence $\Omega_T$ is $\Theta$-regular, in the sense of Definition \ref{defn_abs_cont}. \enex

As it is clear, this example gives a strong indication on conditions that a sesquilinear form $\Omega$ must satisfy to be represented as $\Omega=\Omega_T$. In fact, we can now state the following Radon-Nikodym-like theorem.

\begin{thm}\label{thm_positivedom}
Let $\Omega$, $\Theta$ be sesquilinear forms on $\D\times \D$, with $\Theta$ positive.
The following statements are equivalent.
\begin{itemize}
\item[(i)] $\Omega$ is $\Theta$-regular.
\item[(ii)] There exists a positive self-adjoint operator $H$, with $j_{\sss{\Theta}}(\D) \subset D(H)$, and a linear operator $Y:j_{\sss{\Theta}}(\D)\to D(H)$ such that
\begin{equation}\label{eqn_fundamental} \Omega(\xi,\eta)= \ip{HY j_{\sss{\Theta}}(\xi)}{H j_{\sss{\Theta}}(\eta)}_{\sss{\Theta}} , \quad \forall \xi, \eta \in \D\end{equation}
and the positive sesquilinear form $\Gamma$ defined by
$$ \Gamma(\xi,\eta)=\ip{Hj_{\sss{\Theta}}(\xi)}{Hj_{\sss{\Theta}}(\eta)}_{\sss{\Theta}}+\ip{HYj_{\sss{\Theta}}(\xi)}{HYj_{\sss{\Theta}}(\eta)}_{\sss{\Theta}}, \quad  \xi, \eta \in \D$$
is $\Theta$-absolutely continuous.
\end{itemize}
\end{thm}

\begin{proof}
(i)$\Rightarrow$(ii):\;Let $\Psi\in M(\Omega)$ satisfy the conditions (i) and (ii) of Definition \ref{defn_abs_cont}.\; Clearly, $\Theta+ \Psi$ dominates both $\Theta$ and $\Psi$. Then by Lemma \ref{lemma_oone} there exists a positive self-adjoint operator $C\in{\mc B}(\H_{\sss{\Theta + \Psi}})$, $0\leq C\leq I$, such that
$$ \Theta(\xi,\eta)= \ip{Cj_{\sss{\Theta + \Psi}}(\xi)}{j_{\sss{\Theta + \Psi}}(\eta)}_{\sss{\Theta + \Psi}}, \quad \forall \xi, \eta \in \D.$$
Put $B:= C^{1/2}$. Then, $0\leq B\leq I$ and
\begin{equation}
\Theta(\xi,\eta) = \ip{B j_{\sss{\Theta + \Psi}}(\xi)}{ B j_{\sss{\Theta + \Psi}}(\eta)}_{\sss{\Theta + \Psi}},\;\forall \xi,\eta \in \D,
\label{eq:9.4}
\end{equation}
Since $(\Theta+\Psi)(\xi,\eta)= \ip{j_{\sss{\Theta + \Psi}}(\xi)}{j_{\sss{\Theta + \Psi}}(\eta)}_{\sss{\Theta + \Psi}}$, we also have
$$ \Psi(\xi, \eta)=(\Theta+\Psi)(\xi, \eta)- \Theta(\xi, \eta)=\ip{(I-B^2)j_{\sss{\Theta + \Psi}}(\xi)}{j_{\sss{\Theta + \Psi}}(\eta)}_{\sss{\Theta + \Psi}}.$$
By \eqref{eq:9.4}  an isometry $U$ of $\H_{\sss{\Theta}}$ into $\H_{\sss{\Theta + \Psi}}$ can be defined by putting first
$$
U j_{\sss{\Theta}}(\xi)= B j_{\sss{\Theta + \Psi}}(\xi), \; \xi \in \D
$$
and then extending it to $\H_{\sss{\Theta}}$.

Set  now
\begin{align*}
& S_n = \int^1_{\frac{1}{n}} t^{-1} \sqrt{1-t^2} \, dE(t), \\
& K_n = U\ha   S_n U, \; n \in \NN,
\end{align*}
where $B= \int^1_0 t \, dE(t)$ is the spectral resolution of $B$. Then it is not difficult to show that $\{ K_n \}$
is an increasing sequence of positive operators. We have
\begin{equation}
\lim_{n \rightarrow \infty} K_n j_{\sss{\Theta}}(\xi) \text{ exists in } \H_{\sss{\Theta}},   \; \forall \, \xi \in \D
\label{eq:9.5}
\end{equation}
and
\begin{equation}
\Psi(\xi,\eta) = \lim_{n \to \infty} \ip{K_n j_{\sss{\Theta}}(\xi)}{ K_n j_{\sss{\Theta}}(\eta) }_{\sss{\Theta}},   \; \forall \, \xi,\eta
\in \D.
\label{eq:9.6}
\end{equation}
Indeed, taking into account that since $UU\ha$ is the projection onto the range $\ran U$ of $U$ and that $\ran U=\overline{\ran B}=(\ker B)^\perp $, then $UU\ha$ commutes with $B$ and with $S_n$, we have  for $m>n$,
\begin{align*} \|K_n j_{\sss{\Theta}}(\xi) - K_m j_{\sss{\Theta}}(\xi)\|_{\sss{\Theta}}^2 &=\|U\ha S_n U j_{\sss{\Theta}}(\xi) - U\ha S_m Uj_{\sss{\Theta}}(\xi)\|_{\sss{\Theta}}^2\\
&=\|UU\ha S_n U j_{\sss{\Theta}}(\xi) - UU\ha S_m Uj_{\sss{\Theta}}(\xi)\|_{\sss{\Theta+\Psi}}^2\\
&=\|S_nBj_{\sss{\Theta + \Psi}}(\xi) - S_mBj_{\sss{\Theta + \Psi}}(\xi)\|_{\sss{\Theta + \Psi}}^2\\
&= \int_{1/m}^{1/n} {(1-t^2)} \, d\ip{E(t) j_{\sss{\Theta + \Psi}}(\xi)}{j_{\sss{\Theta + \Psi}}(\xi)}_{\sss{\Theta + \Psi}}\\
&\leq \|(E(1/n)-E(1/m)) j_{\sss{\Theta + \Psi}}(\xi))\|_{\sss{\Theta + \Psi}}^2\to 0, \;\mbox{as}\;n,m\to \infty.
\end{align*}

Now we prove \eqref{eq:9.6}. We denote by $P_{\sss{\Theta + \Psi}}$ the projection of $\H_{\sss{\Theta + \Psi}}$ onto $\ker B$.  Then, for every $\xi, \eta \in \D$, we get
\begin{align}\label{eqn_definingseq}
\lim_{n \to \infty} &\ip{K_n j_{\sss{\Theta}}(\xi)}{ K_n j_{\sss{\Theta}}(\eta)}_{\sss{\Theta}} = \lim_{n\to\infty}\ip{U\ha S_nUj_{\sss{\Theta}}(\xi)}{U\ha S_nUj_{\sss{\Theta}}(\eta)}_{\sss{\Theta}}\\
&=\lim_{n\to\infty}\ip{UU\ha S_nUj_{\sss{\Theta}}(\xi)}{S_nUj_{\sss{\Theta}}(\eta)}_{\sss{\Theta + \Psi}}\nonumber\\
&=\lim_{n\to\infty}\ip{S_nUj_{\sss{\Theta}}(\xi)}{S_nUj_{\sss{\Theta}}(\eta)}_{\sss{\Theta + \Psi}}\nonumber\\
&=\lim_{n\to\infty} \ip{S_nBj_{\sss{\Theta + \Psi}}(\xi)}{S_nBj_{\sss{\Theta + \Psi}}(\eta)}_{\sss{\Theta + \Psi}}\nonumber\\
&=\lim_{n\to\infty}\int_{1/n}^1 (1-t^2)d\ip{E(t) j_{\sss{\Theta + \Psi}}(\xi)}{j_{\sss{\Theta + \Psi}}(\eta)}_{\sss{\Theta + \Psi}}\nonumber\\
&= \ip{(I-B^2)j_{\sss{\Theta + \Psi}}(\xi)}{j_{\sss{\Theta + \Psi}}(\eta)}_{\sss{\Theta + \Psi}}- \ip{P_{\sss{\Theta + \Psi}} j_{\sss{\Theta + \Psi}}(\xi)}{j_{\sss{\Theta + \Psi}}(\eta)}_{\sss{\Theta + \Psi}}\nonumber \\
&=\Psi(\xi,\eta)-\ip{P_{\sss{\Theta + \Psi}} j_{\sss{\Theta + \Psi}}(\xi)}{j_{\sss{\Theta + \Psi}}(\eta)}_{\sss{\Theta + \Psi}}\nonumber .
\end{align}
Now we use the $\Theta$-absolute continuity of $\Psi$. Let $\xi \in \D$ and consider $P_{\sss{\Theta + \Psi}} j_{\sss{\Theta + \Psi}}(\xi)$. Let $\{\xi_n\}$ be a sequence of vectors of $\D$ such that $j_{\sss{\Theta + \Psi}}(\xi_n) \to P_{\sss{\Theta + \Psi}}j_{\sss{\Theta + \Psi}}(\xi)$. Then we have
$$ \lim_{n\to\infty} U j_{\sss{\Theta}}(\xi_n) = \lim_{n\to\infty} B j_{\sss{\Theta + \Psi}}(\xi_n)= BP_{\sss{\Theta + \Psi}}j_{\sss{\Theta + \Psi}}(\xi)=0,$$
so $\|j_{\sss{\Theta}}(\xi_n)\|_{\sss{\Theta}}\to 0$
and
$$\|j_{\sss{\Psi}}(\xi_n)-j_{\sss{\Psi}}(\xi_m)\|^2_{\sss{\Psi}}= \ip{(I-B^2)j_{\sss{\Theta + \Psi}}(\xi_n - \xi_m)}{j_{\sss{\Theta + \Psi}}(\xi_n - \xi_m)}_{\sss{\Theta + \Psi}}\to 0\; $$
as $n,m\to + \infty$.
By the closability of the map $j_{\sss{\Theta}} (\xi) \to j_{\sss{\Psi}}(\xi)$, $\xi\in \D$ it follows that $\|j_{\sss{\Psi}}(\xi_n)\|_{\sss{\Psi}}\to 0$.
Hence,
\begin{align*} \ip{P_{\sss{\Theta + \Psi}} j_{\sss{\Theta + \Psi}}(\xi)}{j_{\sss{\Theta + \Psi}}(\xi)}_{\sss{\Theta + \Psi}}&= \ip{(I - B^2)P_{\Theta + \Psi} j_{\sss{\Theta + \Psi}}(\xi)}{j_{\sss{\Theta + \Psi}}(\xi)}_{\sss{\Theta + \Psi}}\\&=\lim_{n\to\infty}\ip{(I-B^2)j_{\sss{\Theta + \Psi}}(\xi_n)}{j_{\sss{\Theta + \Psi}}(\xi_n)}_{\sss{\Theta + \Psi}} \\ &= \lim_{n\to\infty}\ip{j_{\sss{\Psi}} (\xi_n)}{ j_{\sss{\Psi}} (\xi_n)}_{\sss{\Psi}} =0.
\end{align*}
We put
$$
\begin{cases}
D(K_0) = \{ \zeta\in \H_{\sss{\Theta}}; \dis \lim_{n \rightarrow \infty} K_n \zeta  \text{ exists in } \H_{\sss{\Theta}} \}, \\
K_0 \zeta = \dis \lim_{n \rightarrow \infty} K_n \zeta, \quad \xi \in \D(K_0).
\end{cases}
$$
Then it follows from \eqref{eq:9.5} and \eqref{eq:9.6} that $K_0$ is a positive operator in $\H_{\sss{\Theta}}$ such that
 $D(K_0) \supset j_{\sss{\Theta}}(\D)$. Its Friedrichs extension $K$ is then a positive self-adjoint operator
satisfying the following conditions ({\sf k}$_1$) and ({\sf k}$_2$):
\begin{itemize}
\item[({\sf k}$_1$)]
$j_{\sss{\Theta}}(\D) \subset D(K)$;

\item[({\sf k}$_2$)] $\Psi(\xi,\eta)= \ip{K j_{\sss{\Theta}}(\xi)}{K j_{\sss{\Theta}}(\eta)}_{\sss{\Theta}}, \; \forall\xi,\eta \in \D$.
\end{itemize}

The $\Theta$-absolute continuity of $\Psi$ implies that the map
\begin{equation}\label{eqn_map_u} \emb: \tilde{\xi}=\lim_{n\to\infty} j_{\sss{\Theta +\Psi}}(\xi_n)\in \H_{\sss{\Theta+\Psi}}\mapsto \emb(\tilde{\xi}):=\lim_{n\to\infty} j_{\sss{\Theta}}(\xi_n)\in \H_{\sss{\Theta}}\end{equation}
is well-defined, injective, continuous and has dense range. Thus
$\H_{\sss{\Theta+\Psi}}$ can be identified with a dense subspace of $\H_{\sss{\Theta}}$.

As for the inner product we have
$$\ip{\tilde{\xi}}{\tilde{\eta}}_{\sss{\Theta+\Psi}}=\ip{\emb(\tilde{\xi})}{\emb(\tilde{\eta})}_{\sss{\Theta}}+\ip{K\emb(\tilde{\xi})}{K\emb(\tilde{\eta})}_{\sss{\Theta}}, \quad \forall \tilde{\xi}, \tilde{\eta}\in \H_{\sss{\Theta+\Psi}}. $$
Now define $H=(I+K^2)^{1/2}$. Then $D(H)=D(K)$ and
\begin{equation}\label{eqn_secondkato}\ip{\tilde{\xi}}{\tilde{\eta}}_{\sss{\Theta+\Psi}}=\ip{H\emb(\tilde{\xi})}{H\emb(\tilde{\eta})}_{\sss{\Theta}}, \quad \forall \tilde{\xi}, \tilde{\eta}\in \H_{\sss{\Theta+\Psi}}. \end{equation}
 We set
$$ \Omega_0(j_{\sss{\Theta+\Psi}}(\xi), j_{\sss{\Theta+\Psi}}(\eta)):= \Omega(\xi,\eta), \quad \xi, \eta \in \D.$$
Then $\Omega_0$ is a well-defined sesquilinear form on $j_{\sss{\Theta+\Psi}}(\D)\times j_{\sss{\Theta+\Psi}}(\D)$.

Then by \eqref{eq:domination} it follows that $\Omega_0$ is bounded on $j_{\sss{\Theta+\Psi}}(\D)$ and thus it extends to a bounded sesquilinear form $\tilde\Omega_0$ on $\H_{\sss{\Theta+\Psi}}$. Hence, there exists an operator $Y_{\sss{\Theta+\Psi}} $, bounded on $\H_{\sss{\Theta+\Psi}}$, such that
$$ \Omega(\xi,\eta)=\tilde\Omega_0(j_{\sss{\Theta+\Psi}}(\xi), j_{\sss{\Theta+\Psi}}(\eta))= \ip{Y_{\sss{\Theta+\Psi}} j_{\sss{\Theta+\Psi}}(\xi)}{j_{\sss{\Theta+\Psi}}(\eta)}_{\sss{\Theta+\Psi}}, \quad \forall \xi, \eta \in \D .$$
On the other hand, by \eqref{eqn_secondkato}, we have
\begin{align}\label{eqn_ten}\ip{Y_{\sss{\Theta+\Psi}} j_{\sss{\Theta+\Psi}}(\xi)}{j_{\sss{\Theta+\Psi}}(\eta)}_{\sss{\Theta+\Psi}}&=\ip{H \emb(Y_{\sss{\Theta+\Psi}} j_{\sss{\Theta+\Psi}}(\xi))}{H\emb(j_{\sss{\Theta+\Psi}}(\eta))}_{\sss{\Theta}}\\ \nonumber
&=\ip{H \emb(Y_{\sss{\Theta+\Psi}} j_{\sss{\Theta+\Psi}}(\xi))}{H j_{\sss{\Theta}}(\eta)}_{\sss{\Theta}} , \quad \forall \xi, \eta \in \D.\end{align} Hence,
\begin{equation}\label{eqn_10} \Omega(\xi, \eta) = \ip{H \emb(Y_{\sss{\Theta+\Psi}} j_{\sss{\Theta+\Psi}}(\xi))}{H j_{\sss{\Theta}}(\eta)}_{\sss{\Theta}} , \quad \forall \xi, \eta \in \D.\nonumber\end{equation}
We now define an operator $Y$ from $j_{\sss\Theta}(\D)$ into $\H_{\sss\Theta}$ by putting $Y j_{\sss{\Theta}}(\xi)= \emb(Y_{\sss{\Theta+\Psi}} j_{\sss{\Theta+\Psi}}(\xi))$. The $\Theta$-absolute continuity of $\Psi$ implies that $Y$ is well-defined. Thus finally we have
\begin{equation}\label{eqn_final_abscont} \Omega(\xi,\eta)= \ip{H Yj_{\sss{\Theta}}(\xi))}{H j_{\sss{\Theta}}(\eta)}_{\sss{\Theta}} , \quad \forall \xi, \eta \in \D.\end{equation}
We now prove that $\Gamma$ is $\Theta$-absolutely continuous.
We first observe that using \eqref{eqn_secondkato},  the following inequality can be proven:
\begin{align}\label{eqn_eleven} \|Yj_{\sss{\Theta}}(\xi)\|_{\sss{\Theta}}&= \|\emb (Y_{\sss{\Theta+\Psi}} j_{\sss{\Theta+\Psi}}(\xi)\|_{\sss{\Theta}}\\ & \leq \gamma\|Y_{\sss{\Theta+\Psi}} j_{\sss{\Theta+\Psi}}(\xi)\|_{\sss{\Theta+\Psi}}= \gamma \|HYj_{\sss{\Theta}}(\xi)\|_{\sss{\Theta}}, \quad \forall \xi \in \D.\nonumber\end{align}
 Let $\{\xi_n\}$ be a sequence in $\D$ such that \begin{align*} &\|j_{\sss{\Theta}}(\xi_n)\|_{\sss{\Theta}}\to 0, , \mbox{ as } n \to \infty, \mbox{ and } \\ & \|Hj_{\sss{\Theta}}(\xi_n-\xi_m)\|_{\sss{\Theta}}^2+\|HYj_{\sss{\Theta}}(\xi_n-\xi_m)\|_{\sss{\Theta}}^2\to 0, \mbox{ as } n,m \to \infty. \end{align*}
The closedness of $H$ implies that $\|Hj_{\sss{\Theta}}(\xi_n)\|_{\sss{\Theta}}\to 0$. But, by  \eqref{eqn_secondkato}, $\|Hj_{\sss{\Theta}}(\xi_n)\|_{\sss{\Theta}}=\|j_{\sss{\Theta+\Psi}}(\xi_n)\|_{\sss{\Theta+\Psi}}\to 0$ and thus, by the boundedness of $Y_{\sss{\Theta+\Psi}}$ in $\H_{\sss{\Theta+\Psi}} $ we conclude that $\|Y_{\sss{\Theta+\Psi}}j_{\sss{\Theta+\Psi}}(\xi_n)\|_{\sss{\Theta+\Psi}}\to 0$.
On the other hand, since
$$ \|Yj_{\sss{\Theta}}(\xi_n-\xi_m)\|_{\sss{\Theta}}\leq \gamma \|HYj_{\sss{\Theta}}(\xi_n-\xi_m)\|_{\sss{\Theta}}\to 0,$$
the sequence $\{Yj_{\sss{\Theta}}(\xi_n) \}$ converges to some vector $\zeta\in \H_{\sss{\Theta}}$ and, again by the closedness of $H$, we have $\zeta \in D(H)$ and $HYj_{\sss{\Theta}}(\xi_n)\to H\zeta$. By \eqref{eqn_eleven}, we obtain $\zeta=0$. Using again the closedness of $H$ we finally get $\|HYj_{\sss{\Theta}}(\xi_n)\|_{\sss{\Theta}}\to 0$. Hence $\Gamma$ is $\Theta$-absolutely continuous.
The implication (ii)$\Rightarrow$(i) is obvious once one takes $\Psi=\Gamma$.
\end{proof}

{  \berem The operators $H, Y$ which appear in the representation \eqref{eqn_fundamental} of $\Omega$ depend, clearly, on the choice of the $\Theta$-absolutely continuous form $\Psi\in \M(\Omega)$ (if any, of course). They are not even uniquely determined for a fixed $\Psi \in \M(\Omega)$, but $Y$, depending only on the inner product defined by $\Psi$, is unique. It is worth remarking that following \cite[Ch. VI, Lemma 3.1]{kato} and under the same assumptions of Theorem \ref{thm_positivedom}, it can be shown that there exists a bounded operator $S$ in $\H_{\sss \Theta}$ such that
$$ \Omega(\xi,\eta)= \ip{SH j_{\sss{\Theta}}(\xi)}{H j_{\sss{\Theta}}(\eta)}_{\sss{\Theta}} , \quad \forall \xi, \eta \in \D.$$
The values of the operator $S$ can be, however, arbitrarily chosen on the subspace $(\ran H)^\perp$ of $\H_{\sss{\Theta}}$; so that, even for fixed $\Psi$ and $H$, the uniqueness of the representation is lost.
\enrem
}
\medskip
In \cite{sebestyen} Sebesty\'en and Titkos gave a Radon-Nikodym type theorem for positive sesquilinear form $\Psi$ on $\D\times \D$, with $\Psi$ {\em almost dominated} by $\Theta$; this means that there exists a nondecreasing sequence ${\Psi_n}$ such that $\Theta$ dominates every $\Psi_n$, $n\in {\mb N}$, and $\Psi= \sup_{n\in {\mb N}}\Psi_n$. Along the previous proof we have explicitly constructed this sequence. Actually, the following statement holds (see also \cite[Theorem 3.8]{hassi}).

\begin{cor} Let $\Theta$ and $\Psi$ be positive sesquilinear forms on $\D\times \D$. The following statements are equivalent.
\begin{itemize}
\item[(i)] $\Psi$ is $\Theta$-absolutely continuous.
\item[(ii)] $\Psi$ is almost dominated by $\Theta.$
\end{itemize}
\end{cor}
\begin{proof}
(i)$\Rightarrow$(ii):\; Set $$\Psi_n(\xi, \eta):= \ip{K_n j_{\sss{\Theta}}(\xi)}{ K_n j_{\sss{\Theta}}(\eta) }_{\sss{\Theta}},   \; \forall \, \xi,\eta
\in \D,$$ where the ${K_n}'s$ are the bounded operators in $\H_{\sss{\Theta}}$ defined in the proof of (i)$\Rightarrow$(ii) of Theorem \ref{thm_positivedom}. Then it is easily seen that the nondecreasing sequence $\{\Psi_n\}$ consists of $\Theta$-dominated forms and, as shown in the same proof, $\Psi=\sup_{n \in {\mb N}} \Psi_n$. Hence $\Psi$ is almost dominated by $\Theta$.

(ii)$\Rightarrow$(i):\; This follows from \cite[Theorem 3.8]{hassi} and Theorem \ref{thm_positivedom}.
\end{proof}

\section{Lebesgue-like decomposition}\label{sect_lebesgue}
At this point of our discussion it is natural to pose the question as to whether a Lebesgue-like decomposition holds for a sesquilinear form $\Omega$ for which $\M(\Omega)$ is nonempty. Before considering this question we need to precise the notion of {\em singular} form.
\bedefi Let $\Theta$, $\Omega$  be sesquilinear forms on $\D\times \D$, with $\Theta$ positive. We say that $\Omega$ is $\Theta$-singular if, for every $\xi \in \D$, there exists a sequence $\{\xi_n\}$ in $\D$ such that $$\lim_{n\to \infty}\Theta(\xi_n,\xi_n) =0\;\mbox{ and }\;\lim_{n\to \infty}\Omega(\xi_n-\xi, \xi_n-\xi)=0.$$

\findefi

The previous definition in the case of positive sesquilinar forms coincides with the traditional one.
\begin{prop}\label{prop_28}
Let $\Psi$, $\Theta$ be two positive sesquilinear forms on $\D\times \D$. Then the following assertions are equivalent:\\
(i) $\Psi$ is $\Theta$-singular;\\
(ii)If $\sigma$ is a positive sesquilinear form with $\sigma\leq\Psi$ and $\sigma\leq\Theta$, then $\sigma =0$.
\end{prop}
The previous proposition shows the symmetry of the notion of singularity for positive sesquilinear forms (by exchanging the roles of $\Theta$ and $\Psi$) which is lost if one of them is not necessarily positive.

The following Theorem \ref{thm_lebesgue} gives a variant of the Lebesgue decomposition theorem.

\begin{thm}\label{thm_lebesgue} Let $\Theta$, $\Omega$ be sesquilinear forms on $\D\times \D$, with $\Theta$ positive. If $\M(\Omega)$ is nonempty, then $\Omega$ can be decomposed into
$$ \Omega= \Omega_r +\Omega_s$$
where $\Omega_r$ is $\Theta$-regular and $\Omega_s$ is $\Theta$-singular.

\begin{proof} We choose $\Psi \in \M(\Omega)$ and follow essentially the proof of Theorem \ref{thm_positivedom}. The construction of the sequence $\{K_n\}$ does not depend in fact on the $\Theta$-absolute continuity of $\Omega$ required there. Hence from \eqref{eqn_definingseq} we get the representation
$$ \Psi(\xi,\eta)= \ip{Kj_{\sss{\Theta }}(\xi)}{Kj_{\sss{\Theta }}(\eta)}_{\sss{\Theta}}+\ip{P_{\sss{\Theta + \Psi}} j_{\sss{\Theta + \Psi}}(\xi)}{j_{\sss{\Theta + \Psi}}(\eta)}_{\sss{\Theta + \Psi}},$$
which already proves the statement for a positive sesquilinear form $\Psi$, once one proves that the second term of the right hand side of the previous equation is a singular form (this will be done later).
The kernel of the map ${\sf u}$ defined in \eqref{eqn_map_u}, if $\Omega$ is not $\Theta$-regular, coincides with the kernel of the operator $B$ defined in \eqref{eq:9.4}. Hence \eqref{eqn_secondkato} becomes
\begin{align*}
\ip{\tilde{\xi}}{\tilde{\eta}}_{\sss{\Theta+\Psi}}&=\ip{H\emb(\tilde{\xi})}{H\emb(\tilde{\eta})}_{\sss{\Theta}} \\ &+\ip{P_{\sss{\Theta + \Psi}} \tilde{\xi}}{P_{\sss{\Theta + \Psi}}\tilde\eta}_{\sss{\Theta + \Psi}}, \quad \forall \tilde{\xi}, \tilde{\eta}\in \H_{\sss{\Theta+\Psi}}. \end{align*}
Thus \eqref{eqn_10} reads as follows
\begin{align*} \Omega(\xi, \eta) &= \ip{H \emb(Y_{\sss{\Theta+\Psi}} j_{\sss{\Theta+\Psi}}(\xi))}{H j_{\sss{\Theta}}(\eta)}_{\sss{\Theta}}\\&+\ip{P_{\sss{\Theta + \Psi}} Y_{\sss{\Theta+\Psi}} j_{\sss{\Theta+\Psi}}(\xi)}
{P_{\sss{\Theta + \Psi}}j_{\sss{\Theta+\Psi}}(\eta)}_{\sss{\Theta + \Psi}} , \quad \forall \xi, \eta \in \D.
\end{align*}
We now define an operator $Z$ from $j_{\sss\Theta}(\D)$ into $\H_{\sss\Theta}$ by putting $Z j_{\sss{\Theta}}(\xi)= \emb(Y_{\sss{\Theta+\Psi}}(I-P_{\sss{\Theta+\Psi}}) j_{\sss{\Theta+\Psi}}(\xi))$. Since $j_{\sss{\Theta}}(\xi)=0$ if and only if $j_{\sss{\Theta+\Psi}}(\xi)\in \ker B$,   $Z$ is well-defined.
Now we define, for every $\xi, \eta \in \D$,
\begin{align} \label{eqn_reg}\Omega_r(\xi, \eta) &= \ip{H Z j_{\sss{\Theta}}(\xi))}{H j_{\sss{\Theta}}(\eta)}_{\sss{\Theta}}; \\
\label{eqn_sing} \Omega_s(\xi, \eta)& = \ip{H\emb(Y_{\sss{\Theta+\Psi}}P_{\sss{\Theta+\Psi}} j_{\sss{\Theta+\Psi}}(\xi))}{H j_{\sss{\Theta}}(\eta)}_{\sss{\Theta}}\\ &+ \ip{P_{\sss{\Theta + \Psi}} Y_{\sss{\Theta+\Psi}} j_{\sss{\Theta+\Psi}}(\xi)}
{P_{\sss{\Theta + \Psi}}j_{\sss{\Theta+\Psi}}(\eta)}_{\sss{\Theta + \Psi}}.\nonumber\end{align}
It is clear that $\Omega=\Omega_r+\Omega_s$. It remains to prove that $\Omega_r$ and $\Omega_s$ have the desired properties. The $\Theta$-regularity of $\Omega_r$ can be proved in the very same way of the corresponding proof for $\Omega$ at the end of the proof of Theorem \ref{thm_positivedom} and we omit the details. As for $\Omega_s$, let $\xi \in \D$ and $\{\xi_n\}$ a sequence in $\D$ such that $j_{\sss{\Theta+\Psi}}(\xi_n)\to P_{\sss{\Theta + \Psi}}j_{\sss{\Theta+\Psi}}(\xi)$. Then, as in Theorem \ref{thm_positivedom}, $\|j_{\sss{\Theta}}(\xi_n)\|_{\sss{\Theta}}\to 0$.
Then we have
\begin{align*} \|H\emb(Y_{\sss{\Theta+\Psi}}P_{\sss{\Theta+\Psi}} j_{\sss{\Theta+\Psi}}(\xi-\xi_n)) \|_{\sss{\Theta}}^2&= \|Y_{\sss{\Theta+\Psi}}P_{\sss{\Theta+\Psi}} j_{\sss{\Theta+\Psi}}(\xi-\xi_n)
\|_{\sss{\Theta+\Psi}}^2\\ &-\|P_{\sss{\Theta + \Psi}}Y_{\sss{\Theta+\Psi}}P_{\sss{\Theta+\Psi}}j_{\sss{\Theta+\Psi}}(\xi-\xi_n)\|_{\sss{\Theta+\Psi}}^2\end{align*}
and both terms of the right hand side tend to 0, by the boundedness of $Y_{\sss{\Theta+\Psi}}$ and by the definition of $\{\xi_n\}$.

Moreover
\begin{align*}\| H j_{\sss{\Theta}}(\xi-\xi_n) \|_{\sss{\Theta}}^2&=\|j_{\sss{\Theta+\Psi}}(\xi-\xi_n)\|_{\sss{\Theta+\Psi}}^2 - \|P_{\sss{\Theta + \Psi}}j_{\sss{\Theta+\Psi}}(\xi-\xi_n)\|_{\sss{\Theta+\Psi}}^2\\
&\to \|j_{\sss{\Theta+\Psi}}(\xi)-P_{\sss{\Theta + \Psi}}j_{\sss{\Theta+\Psi}}(\xi)\|_{\sss{\Theta+\Psi}}^2.
\end{align*}
A simple application of the Cauchy-Schwarz inequality shows then that
$$ \ip{H\emb(Y_{\sss{\Theta+\Psi}}P_{\sss{\Theta+\Psi}} j_{\sss{\Theta+\Psi}}(\xi-\xi_n))}{H j_{\sss{\Theta}}(\xi-\xi_n)}_{\sss{\Theta}}\to 0.$$

Finally, since
$$ P_{\sss{\Theta+\Psi}}Y_{\sss{\Theta+\Psi}} j_{\sss{\Theta+\Psi}}(\xi-\xi_n)\to P_{\sss{\Theta+\Psi}}Y_{\sss{\Theta+\Psi}}(I-P_{\sss{\Theta+\Psi}})j_{\sss{\Theta+\Psi}}(\xi)$$
and
$$P_{\sss{\Theta+\Psi}}j_{\sss{\Theta+\Psi}}(\xi-\xi_n)\to 0,$$ we conclude that
$$ \ip{P_{\sss{\Theta+\Psi}}Y_{\sss{\Theta+\Psi}} j_{\sss{\Theta+\Psi}}(\xi-\xi_n)}{P_{\sss{\Theta+\Psi}}j_{\sss{\Theta+\Psi}}(\xi-\xi_n)}_{\sss{\Theta+\Psi}}\to 0.$$
So that
$$\lim_{n\to\infty}\Omega_s(\xi-\xi_n,\xi-\xi_n)=0;$$
that is, $\Omega_s$ is $\Theta$-singular.
\end{proof}

\end{thm}
\berem As in the case of positive sesquilinear forms (see the discussion in \cite[Section 4]{hassi} and also \cite{sebestyen}), the Lebesgue-like decomposition is not unique.
\enrem

As a consequence of the previous theorem we find, clearly, the results obtained in \cite[Theorem 2.11; Proposition 3.7]{hassi}. For the sake of completeness we show explicitly (see also \cite[Theorem 2.2]{simon}) how they can be recovered using the techniques developed in this paper.
\begin{cor}\label{thm_lebesgue2} Let $\Theta$, $\Psi$ be  positive sesquilinear forms on $\D\times \D$. Then $\Psi$ can be decomposed into
$$ \Psi= \Psi_a +\Psi_s$$
where $\Psi_a$ is $\Theta$-absolutely continuous and $\Psi_s$ is $\Theta$-singular.

Moreover, if $\Phi$ is $\Theta$-absolutely continuous and $\Phi\leq \Psi$, then $\Phi\leq \Psi_a$.
\end{cor}

\begin{proof}
Following the same construction made in the proof of  Theorem \ref{thm_positivedom} we define $$\Psi_a(\xi , \eta) = \lim_{n\to +\infty} \ip{K_n j_{\sss{\Theta}} (\xi)} {K_nj_{\sss{\Theta}} (\eta)}_{\sss{\Theta}}$$ $$\Psi_s(\xi, \eta) =\ip{P_{\sss{\Theta+\Psi}}j_{\sss{\Theta+\Psi}} (\xi)}{ j_{\sss{\Theta+\Psi}} (\eta)}_{\sss{\Theta+\Psi}},\quad \xi,\eta\in \D .$$ Then $\Psi_a$ and $\Psi_s$ verify the assertion.

{In order to prove the second statement, let us consider the operator $B$
constructed in the proof of Theorem \ref{thm_positivedom}. We keep the notations introduced there; in particular, we denote by $P_{\sss{\Theta+\Psi}}$ the projection onto ${\rm Ker\,}B$. Let ${I}_{\sss{\Psi, \Phi}}: \H_{\sss{\Theta+\Psi}} \to \H_{\sss{\Theta+\Phi}}$ be the linear map defined first by $${I}_{\sss{\Psi, \Phi}}:\,j_{\sss{\Theta+\Psi}}(\xi)\mapsto j_{\sss{\Theta+\Phi}}(\xi), \quad \xi \in \D.$$
 Then ${I}_{\sss{\Psi, \Phi}}$ is well-defined and contractive and extends to $\H_{\sss{\Theta+\Psi}}$ (we denote this extension by the same symbol). 
  Now, the $\Theta$-absolute continuity of $\Phi$ implies that ${\rm Ker\,}{I}_{\sss{\Psi, \Phi}}={\rm Ker\,}B$. 
The inclusion ${\rm Ker\,}{I}_{\sss{\Psi, \Phi}}\subseteq {\rm Ker\,}B$ is obvious. Suppose that $\zeta \in {\rm Ker\,}B$;  $\zeta =\lim_{n\to \infty}j_{\sss{\Theta+\Psi}}(\xi_n)$, for some sequence $\{\xi_n\} \subset \D$,  and $B\zeta =\lim_{n\to \infty}Bj_{\sss{\Theta+\Psi}}(\xi_n)=0$. Then, $\Theta(\xi_n, \xi_n)\to 0$. Now, since ${I}_{\sss{\Psi, \Phi}}\zeta = \lim_{n\to \infty}{I}_{\sss{\Psi, \Phi}}j_{\sss{\Theta+\Psi}}(\xi_n)=\lim_{n\to \infty}j_{\sss{\Theta+\Phi}}(\xi_n)$, necessarily $\Phi(\xi_n-\xi_m, \xi_n-\xi_m)\to 0$. By the $\Theta$-absolute continuity of $\Phi$, it follows that $\Phi(\xi_n, \xi_n)\to 0$. Hence ${I}_{\sss{\Psi, \Phi}}\zeta=0$.

 Then, we have
 \begin{align*}
( \Theta + \Phi)(\xi,\xi) &=\|j_{\sss{\Theta+\Phi}}(\xi)\|^2_{\sss{\Theta+\Phi}}= \|{I}_{\sss{\Psi, \Phi}}j_{\sss{\Theta+\Psi}}(\xi)\|^2_{\sss{\Theta+\Phi}}\\
&=\|{I}_{\sss{\Psi, \Phi}}(I- P_{\sss{\Theta+\Psi}})j_{\sss{\Theta+\Psi}}(\xi)\|^2_{\sss{\Theta+\Phi}}\\
&\leq \|(I- P_{\sss{\Theta+\Psi}})j_{\sss{\Theta+\Psi}}(\xi)\|^2_{\sss{\Theta+\Psi}}\\
&=\|j_{\sss{\Theta+\Psi}}(\xi)\|^2_{\sss{\Theta+\Psi}}-\ip{P_{\sss{\Theta+\Psi}}j_{\sss{\Theta+\Psi}}(\xi)}{j_{\sss{\Theta+\Psi}}(\xi)}_{\sss{\Theta+\Psi}}\\
&=(\Theta+\Psi)(\xi,\xi)-\Psi_s(\xi,\xi) = (\Theta+\Psi_a)(\xi,\xi).
 \end{align*}

  Thus $\Phi \leq \Psi_a$.}
\end{proof}

\berem The previous statements apply in particular when $\D$ is a pre-Hilbert space with inner product $\ip{\cdot}{\cdot}$, if we take $\Theta$ to be exactly equal to the inner product of $\D$. Let $\H$ be the norm-completion of $\D$. Then it is easy to check that a positive sesquilinear form $\Psi$ on $\D\times \D$ is $\ip{\cdot}{\cdot}$-absolutely continuous if, and only if $\Psi$ is closable in $\H$. Then the results of Theorem \ref{thm_positivedom} ( or, better, of Corollary \ref{thm_lebesgue2}) and Theorem \ref{thm_lebesgue} (or, better, of Corollary \ref{thm_lebesgue2}) reduce to the statements proved by Simon in \cite{simon} on the decomposition of closed (or closable) sesquilinear forms in $\H$. \enrem

\berem \label{rem_47}Let $\Theta$ be a positive sesquilinear form on $\D\times \D$. With obvious modification of current definitions, we say that a  symmetric sesquilinear form $\Omega$ is  {\em $\Theta$-bounded from below} if there exists $c \in {\mb R}$ such that $$ c\,\Theta(\xi,\xi) \leq \Omega(\xi,\xi), \quad \forall\xi\in \D.$$

Clearly, if $\Omega$ is $\Theta$-bounded from below, then $\Omega- c\, \Theta$ is positive.
Similarly, we call a sesquilinear form $\Omega$  on $\D\times \D$,
{\em $\Theta$-sectorial} if there exist $\delta \in {\mb R}$ and $\gamma >0$ such that
\begin{align} \label{eq_one}
&\; \Re\Omega(\xi,\xi) \geq \delta \Theta(\xi,\xi), \quad \forall \xi \in \D;\\
&| \Im \Omega(\xi, \xi)|\leq \gamma (\Re\Omega(\xi,\xi)-\delta \Theta(\xi,\xi)) , \quad \forall \xi \in \D.\nonumber
\end{align}

From the definition itself it follows that

$$|(\Omega -\delta \Theta)(\xi, \eta)|\leq (1+\gamma)((\Re\Omega-\delta \Theta)(\xi,\xi))^{1/2} (\Re \Omega-\delta \Theta)(\eta,\eta))^{1/2}, \quad \forall \xi, \eta \in \D;$$
this clearly means that $(1+\gamma)(\Re\Omega-\delta \Theta) \in \M(\Omega)$.

\enrem
For $\Theta$-sectorial forms, the foregoing results produce the following characterization whose easy proof will be omitted.
\begin{prop} \label{prop_sectorial}Let $\Omega$ be $\Theta$-sectorial. The following statements are equivalent.
\begin{itemize}
\item[(i)] The positive sesquilinear form $\Re\Omega -\delta \Theta$ is $\Theta$-absolutely continuous.
\item[(ii)] For every sequence $\{\xi_n\}$ in $\D$ such that
$$ \Theta(\xi_n, \xi_n)\to 0 \mbox{ and } \Re\Omega (\xi_n - \xi_m, \xi_n - \xi_m) \to 0,$$
$\Re\Omega(\xi_n, \xi_n) \to 0$ results.
\item[(iii)] $\Omega$ is $\Theta$-regular.
\end{itemize}
\end{prop}
\berem If $\D$ is a pre-Hilbert space and $\Theta(\cdot, \cdot)= \ip{\cdot}{\cdot}$ we simply call {\em sectorial} a $\ip{\cdot}{\cdot}$-sectorial form. In this case condition (ii) of the previous Proposition, simply says that
$\Re\Omega$  is closable; this means that $\Omega$ is closable in the sense of \cite{kato}.
\enrem

\medskip

If $\Omega$ is a $\Theta$-regular sesquilinear form on $\D\times \D$, taking into account the representation
$$ \Omega(\xi,\eta)= \ip{HY j_{\sss{\Theta}}(\xi)}{H j_{\sss{\Theta}}(\eta)}_{\sss{\Theta}} , \quad \forall \xi, \eta \in \D,$$
established in Theorem \ref{thm_positivedom}, it is natural to pose the question as to whether $\Omega$ can also be represented as
\begin{equation} \label{eqn_repOmega} \Omega(\xi,\eta)= \ip{T j_{\sss{\Theta}}(\xi)}{ j_{\sss{\Theta}}(\eta)}_{\sss{\Theta}} \end{equation}
at least when $\xi, \eta$ run onto a sufficiently large subspace of $\D$.
Let us define the following subspace of $\D$
$$\D\sggs:=\{\xi \in \D: HYj_{\sss \Theta}(\xi)\in D(H)\}.$$
Then it is clear that the operator $T:=H^2Y$ is well defined on $\D\sggs$ and \eqref{eqn_repOmega} holds, for every $\xi, \eta \in \D\sggs.$
Nevertheless, $\D\sggs$ can be very poor and since no topology is given to $\D$, the possibility of controlling the size of $\D\sggs$ seems to be hopeless. For this reason, we will confine this analysis (Section \ref{sect_solvable})  by considering $\D$ as a dense subspace of a Hilbert space $\H$, with inner product $\ip{\cdot}{\cdot}$ and we choose $\Theta(\cdot, \cdot)= \ip{\cdot}{\cdot}$.

\section{Solvable forms in Hilbert space} \label{sect_solvable}

In what follows we need the notion of Banach-Gelfand triplet which we recall for reader's convenience.
{Let $\H$ be a Hilbert space (with inner product $\ip{\cdot}{\cdot}$ and norm $\|\cdot\|$). Let $\E$ be a dense subspace of $\H$ which is a Banach space with respect to a norm $\|\cdot\|_\E$ defining on $\E$ a topology finer than that induced by the norm of $\H$. In this case, $\H$ can be continuously embedded into the conjugate Banach dual space $\E^\times$. We get in this way the {\em Banach-Gelfand triplet} (a special kind of rigged Hilbert space)
\begin{equation} \label{eqn_BGT}\E[\|\cdot\|_\E] \hookrightarrow \H[\|\cdot\|] \hookrightarrow \E^\times[\|\cdot\|_{\E^\times}],\end{equation}
where $\|\cdot\|_{\E^\times}$ denotes the usual norm of ${\E^\times}$.
If $\E[\|\cdot\|_\E]$ is a {\em reflexive} Banach space, then the embedding of $\H$ into $\E^\times[\|\cdot\|_{\E^\times}$ has {\em dense range}.}

\medskip
{If $\E$ and $\F$ are Banach spaces, we will use the notation ${\mc B}(\E, \F)$ for the vector space of all bounded operators from $\E$ into $\F$. If $\E=\F$ we put ${\mc B}(\E) ={\mc B}(\E, \E)$.

If $\E$ is reflexive and $X \in {\mc B}(\E, \E^\times)$ then, the operator $X^\dag$ (the {\em adjoint} of $X$), defined by
$$ \ip{X^\dag\xi}{\eta}= \overline{\ip{X\eta}{\xi}}, \quad \xi, \eta\in \D,$$ is also a member of ${\mc B}(\E, \E^\times)$.
\bedefi
Let $\Theta$ be a positive sesquilinear form on $\D\times \D$. We say that a norm $\|\cdot\|$ on $\D$ is {\em compatible} with $\Theta$ ({$\Theta$-{\em compatible}, for short}) if the following two conditions are fulfilled:
\begin{itemize}
\item[(s.1)] $\Theta(\xi, \xi) \leq \|\xi\|^2,\quad \forall \xi \in \D$;
\item[(s.2)] If $\{\xi_n\}$ is a sequence in $\D$ such that  $\Theta (\xi_n, \xi_n)\to 0$ and $\|\xi_n- \xi_m\|\to 0$, then $\|\xi_n\| \to 0$.
\end{itemize}
\findefi

If $\Theta$ possesses a compatible norm, then, clearly, ${\sf N}(\Theta)=\{0\}$. Let us denote by $\E $ the Banach space completion of $\D [\|\cdot\|] $.
Taking into account (s.1) and (s.2), it turns out that $\E$ can be identified with a dense subspace $D(\overline{\Theta})$ of $\H_{\sss\Theta}$ and so we can construct a Banach-Gelfand triplet
\begin{equation}\label{triplet2} \E \hookrightarrow \H_{\sss\Theta} \hookrightarrow \E ^\times\end{equation}
in standard way; as usual $\E ^\times$ denotes the conjugate dual of $\E $. The dual norm of $\E ^\times$ will be denoted by $\|\cdot\|^\times$.   We will assume that the form which puts $\E $ and $\E ^\times$ in duality is an extension of the inner product of $\H_{\sss{\Theta}}$. So that, if $\Lambda \in \E ^\times$ and $\hat \xi \in \E $ we may also write $\Lambda (\hat\xi)=\ip{\Lambda}{\hat\xi}_{\sss{\Theta}}$, for indicating the value that the conjugate linear functional $\Lambda$ takes at $\hat\xi$. This assumption also implies that the embedding $\H_{\sss{\Theta}}\hookrightarrow \E ^\times$ can be thought simply as an inclusion.

As announced before, from now on we assume that $\Theta(\cdot,\cdot)=\ip{\cdot}{\cdot}$ the inner product of a Hilbert space $\H$ and omit any reference to $\Theta$ in the notations.

\bedefi \label{defn_q_closed}Let $\H$ be a Hilbert space, with inner product $\ip{\cdot}{\cdot}$ and norm $\|\cdot\|$, and $\ggs$ a sesquilinear form on $\D\times \D$ with $ \D$ dense in $\H$. We say that $\ggs$ is {\em q-closable} if there exists a norm $\|\cdot\|\sggs$ on $ \D$, compatible with $\ip{\cdot}{\cdot}$ with the following properties:

\begin{itemize}
\item[(\sf cl.1)]$\|\xi\|\leq \|\xi\|\sggs, \; \forall \xi \in \D$;
\item[(\sf cl.2)] the completion $\E\sggs$ of $\D[\|\cdot\|\sggs]$ is a {reflexive} Banach space.
\item[(\sf cl.3)] there exists $\beta>0$ such that $|\ggs(\xi, \eta)|\leq \beta \|\xi\|\sggs\|\eta\|\sggs, \; \forall \xi, \eta \in \D$.
\end{itemize}
The form $\ggs$ is {\em q-closed} if $\D[\|\cdot\|\sggs]$ is a {reflexive} Banach space.

\findefi

To every q-closable sesquilinear form it is, therefore, canonically associated a Banach-Gelfand triplet
\begin{equation}\label{triplet} \E\sggs\hookrightarrow \H \hookrightarrow \E\sggs^\times\end{equation}
as described above.

For convenience we put $\gi(\xi,\eta)=\ip{\xi}{\eta}$ if we need to denote the inner product as a sesquilinear form. 

\begin{prop}Every q-closable sesquilinear form $\ggs$ has a q-closed extension $\overline{\Omega}$ in $\H$.
\end{prop}
\begin{proof} The assumption implies that $\E\sggs$ can be identified with a subspace $D(\overline{\Omega})$ of $\H$ and, by ({\sf{ cl.3}}), $\Omega$ is bounded in $\D[\|\cdot\|\sggs]$; thus it extends to $D(\overline{\Omega})$.
\end{proof}
\beex Every densely defined $\iota$-regular sesquilinear form $\Omega$ is q-closable. Indeed, if $\Psi \in \M(\Omega)$ is $\iota$-absolutely continuous, then one can choose, for instance, $\|\xi\|\sggs = (\|\xi\|^2 + \|\xi\|_{\sss \Psi}^2)^{1/2}$ and verify easily the conditions of Definition \ref{defn_q_closed}.
\enex
From now on, we confine ourselves to consider q-closed sesquilinear forms on $\D \times \D$.
\bedefi \label{defn_39}Let $\ggs$ be a q-closed sesquilinear form defined on $\D\times \D$, with $\D$ a dense subspace of the Hilbert space $\H$. We say that $\ggs$ is {\em solvable} if there exists a sesquilinear form $\gb$, bounded in $\H$, such that
\begin{itemize}
\item[(a.1)]  ${\sf N}(\ggs+\gb)=\{0\}$
\item[(a.2)]For every $\Lambda\in \E\sggs^\times$ there exists  $\xi \in \E\sggs$ such that
$$ \ip{\Lambda}{\eta}= (\ggs+\gb)(\xi, \eta), \quad \forall \eta \in \E\sggs.$$
\end{itemize}
The set of all bounded $\gb$'s satisfying these conditions is denoted by $\mathfrak{P}(\ggs)$.
\findefi

Let $\ggs$ be a q-closed sesquilinear form, $\gb $ a bounded sesquilinear form on $\H \times \H$ and $\ggs\sgb:= \ggs+\gb$. If $\xi \in \D$, we define a conjugate linear functional $\ggs\sgb^\xi$ on $\E\sggs$ by
$$\ip{ \ggs\sgb^\xi}{\eta}=\ggs(\xi,\eta)+\gb(\xi,\eta).$$
Then, $\ggs\sgb^\xi$ is bounded and so $\ggs\sgb^\xi\in \E\sggs^\times$.

Let $X\sgb: \E\sggs \to \E\sggs^\times$ be the linear map on $\E\sggs$ defined by $X\sgb\xi =\ggs\sgb^\xi$.
Then,  $X\sgb \in {\mc B}(\E\sggs, \E\sggs^\times)$ and the following lemma holds.
\begin{lemma} \label{lemma_38} The following statements are equivalent.
\begin{itemize}\item[(i)]$\gb \in \mathfrak{P}(\ggs)$.
\item[(ii)] $X\sgb$ is a bijection of $\E\sggs$ onto $\E\sggs^\times$.
\item[(iii)] $X\sgb$ has a bounded inverse $X\sgb^{-1}: \E\sggs^\times \to \E\sggs$.
\end{itemize}
\end{lemma}
Using Lemma \ref{lemma_38} and James' theorem \cite[Sec. 1.13]{megg} one can prove the following
\begin{prop} Let $\ggs$ be a q-closed sesquilinear form and $\gb $ a bounded sesquilinear form on $\H \times \H$. The following statements are equivalent.
\begin{itemize}
\item[(i)]$\gb \in \mathfrak{P}(\ggs)$.
\item[(ii)] There exist $c_1, c_2, c'_1, c'_2>0$ such that
\begin{itemize}
\item[(ii.a)] for every $\xi \in \E\sggs$ there exists $\bar \eta\in \E\sggs$ such that
$$ c_1\|\xi\|\sggs \leq |(\ggs + \gb)(\xi, \bar\eta)| \leq c_2\|\xi\|\sggs;$$
\item[(ii.b)] for every $\eta \in \E\sggs$ there exists $\bar \xi\in \E\sggs$ such that
$$ c'_1\|\eta\|\sggs \leq |(\ggs + \gb)(\bar\xi, \eta)| \leq c'_2\|\eta\|\sggs.$$
\end{itemize}
\end{itemize}
\end{prop}

\beex
We show that every closed sectorial form, with  domain $\D$ in Hilbert space $\H$, is solvable. Following Remark \ref{rem_47}, we define $\|\xi\|\sggs= (\|\xi\|^2 + \Re\Omega(\xi,\xi)-\delta\|\xi\|^2)^{1/2}$, $\xi\in \D$, and set $\H\sggs=\D[\|\cdot\|\sggs]$.
In this case the triplet \eqref{triplet} consists of Hilbert spaces $\H\sggs\hookrightarrow \H\hookrightarrow \H\sggs^\times$.

Let $\lambda \in {\mb C}$ and suppose $\xi \in {\sf N}(\ggs-\lambda\gi)$. Then $\ggs (\xi, \xi)=\lambda \|\xi\|^2$. Hence, $\|\xi\|\sggs= ({\Re\lambda\|\xi\|^2-\delta \|\xi\|^2 +\|\xi\|^2})^{1/2}$.
 Thus, if $\Re\lambda \leq \delta-1$, we necessarily have $\|\xi\|\sggs= 0$.

For shortness, we put $\ggs_\lambda:= \ggs- \lambda \gi$.
Then, if $\Re\lambda \leq\delta-1$, $\ggs_\lambda$ is bounded and coercive. Indeed,
$$\Re\ggs_\lambda(\xi,\xi)= \Re \ggs(\xi,\xi)-\Re\lambda \|\xi\|^2 = \|\xi\|^2\sggs + (\delta-1-\Re\lambda)\|\xi\|^2\geq \|\xi\|^2\sggs.$$
Hence, by the Lax-Milgram theorem, if $\Lambda\in \H\sggs^\times$, there exists $\xi \in \H\sggs$ such that
$$ \ip{\Lambda}{\eta}= \ggs_\lambda(\xi, \eta), \quad \forall \eta \in \H\sggs.$$
\enex

Now we prove the following result.
\begin{thm}\label{thm_59} Let $\ggs$ be a q-closed solvable sesquilinear form defined on $\D\times\D$, with $\D$ a dense domain in Hilbert space $\H$.
Then, there exists a closed operator $T$ with domain $D(T)\subset \D$, dense in $\H$, such that
$$ \ggs(\xi, \eta)= \ip{T\xi}{\eta}, \quad \forall \xi\in D(T), \eta \in \E\sggs.$$
In particular, if the sesquilinear form $\gb$ of Definition \ref{defn_39} has the form  $\gb =-\lambda \gi$, with $\lambda \in {\mb C}$, then  $\lambda \in \varrho(T)$, the resolvent set of $T$.
\end{thm}
\begin{proof} Let $\gb \in \mathfrak{P}(\ggs)$ and $X\sgb \in {\mc B}(\E\sggs, \E\sggs^\times)$ be defined as above.

By Lemma \ref{lemma_38}, $X\sgb$ has a bounded inverse $X\sgb^{-1}$.
Put $D(S)=\{\xi \in \E\sggs: X\sgb\xi \in \H\}$. Then $D(S)$ is dense in $\H$. Indeed, taking into account that $\E\sggs$ is dense in $\H$, it suffices to show that $D(S)$ is dense in $\E\sggs [\|\cdot\|\sggs]$. Since $\H$ { is dense in } $\E\sggs ^\times$, then, for every $f\in\E\sggs$ there exists $\{g_n\}\subset\H$ with $\|g_n-X\sgb f\|_{\ggs} ^\times\rightarrow 0;$
since $X\sgb^{-1}$ is bounded, it follows that $$D(S)\ni X\sgb^{-1}g_n\stackrel{\|\cdot\|\sggs}{\rightarrow}  f,$$
which proves the statement.

Define $S\xi=X\sgb\xi$, for $\xi \in D(S)$.
It is clear that
$$ \ip{S\xi}{\eta} = \ip{X\sgb\xi}{\eta}=\ggs\sgb(\xi,\eta), \quad \forall \xi \in \D(S),\, \eta \in \E\sggs.$$
Now we want to prove that $S$ is closed in $\H$.
Since $X\sgb\in {\mc B}( \E\sggs, \E\sggs^\times)$, it has an adjoint $X\sgb^\dag \in {\mc B}( \E\sggs, \E\sggs^\times)$.
Moreover, as it is easy to see,  $D(S^*)=\{\eta \in \E\sggs: X\sgb^\dag\eta \in \H\}$.
Hence, $X\sgb^\dag\upharpoonright _{D(S^*)}=S^*$.

Since $X\sgb^\dag$ is invertible, with bounded inverse, in similar way to what done for $D(S)$ one can prove that $D(S^*)$ is dense. By a symmetry argument we can prove that $(S^*)^*=S$. Hence $S$ is closed.
The proof is complete if we define $T$ by putting $D(T)=D(S)$ and $T=S-B$, where $B$ is the unique bounded operator in $\H$ such that $\gb(\xi,\eta)=\ip{B\xi}{\eta}, \forall \xi,\eta \in \H$.

The second statement can be proved as follows.

Let $\gb=-\lambda \gi \in  \mathfrak{P}(\ggs)$, $\lambda \in {\mb C}$. Then, as seen before, $S^{-1}$ is the restriction to $\H$ of $X\sgb^{-1}$.  Now recall that $X\sgb^{-1}$ is continuous from $\E\sggs^\times$ to $\E\sggs$, i.e., there exists $\beta>0$ such that $$\|X\sgb^{-1}\Lambda\|\sggs\leq \beta\|\Lambda\|_{\ggs^\times},\quad\forall\Lambda\in\E\sggs^\times.$$ Then comparing the topologies, we conclude that there exists $\beta'>0$ such that
$$\|S^{-1}f\|_\H\leq \beta'\|f\|_\H ,\quad \forall f \in\H$$ and so $\lambda \in \varrho(T)$.
\end{proof}

 The closed operator $T$ which represents $\ggs$ is not unique, in general.

\berem We point out that the proof of Theorem \ref{thm_59} does not strictly require that $\Omega$ is q-closed in $\D$. In fact, if $\Omega$ is only q-closable, we can replace $\Omega$ with $\overline{\Omega}$ and $\D$ with $D(\overline{\Omega})$ and only small technical modifications are needed in the proof. But, of course, the domain $D(T)$ of the operator $T$ whose existence is claimed in that theorem will be a subspace of $D(\overline{\Omega})$ and thus it might have a very small intersection with the initial domain $\D$.
\enrem

\subsection{Solvability and numerical range} It is of course of particular interest the case where $\mathfrak{P}(\ggs)$ contains {\em scalars}, i.e., for some $\lambda \in {\mb C}$, $-\lambda\gi \in \mathfrak{P}(\ggs)$.

For examining this situation, it is convenient to consider the {\em numerical range} $\gn\sggs$ of $\ggs$, i.e., the set $\gn\sggs=\{\ggs(\xi, \xi); \xi \in \D, \|\xi\|=1\}$.

\begin{thm} Let $\ggs$ be a q-closed sesquilinear form on $\D\times \D$. Assume that the norm $\|\cdot\|\sggs$ which makes $\D$ into a reflexive Banach space $\E[\|\cdot\|\sggs]$ satisfies the following condition
 \begin{itemize}
\item[\sf{(qc)}] If $\{\xi_n\}$ is a sequence in $\D$ such that $\|\xi_n\|\to 0$ and $\ds\lim_{n \to \infty}|\ggs(\xi_n, \xi_n)|=0$, then $\|\xi_n\|\sggs\to 0$.
\end{itemize}

If $\lambda \not\in \overline{\gn\sggs}$, then  $-\lambda\gi \in \mathfrak{P}(\ggs)$.
 \end{thm}
\begin{proof}
If $-\lambda \gi\not\in \mathfrak{P}(\ggs)$, then either $N(\ggs-\lambda\gi) \neq \{0\}$ or (a.2) of Definition \ref{defn_39} is not satisfied.
If $N(\ggs-\lambda\gi) \neq \{0\}$, then there exists $\xi \in \D$, with $\|\xi\|=1$, such that $\ggs(\xi,\eta)-\lambda \ip{\xi}{\eta}=0$, for every $\eta \in \D$. Then, in particular $\ggs(\xi, \xi)=\lambda$. Hence, $ \lambda \in \gn\sggs$.

Now assume that  $-\lambda \gi \not\in \mathfrak{P}(\ggs)$ and $N(\ggs-\lambda\iota) = \{0\}$. Then $X_\lambda:=X_{-\lambda \gi}$ has an inverse which is not everywhere defined in $\E\sggs^\times$. If the range $\ran (X_\lambda)$ of $X_\lambda$ is dense then $X_\lambda^{-1}$ is necessarily unbounded. Then there exists a sequence $\{\xi_n\} \subset \E\sggs$ such that $\|\ggs_\lambda^{\xi_n}\|\sggs^\times =1$, for every $n \in {\mb N}$ and $\|\xi_n\|\sggs \to \infty$, as $n \to \infty$. Thus, the sequence $\{\tau_n\}$ defined by
$$ \tau_n:= \frac{\ggs_\lambda^{\xi_n}}{\|\xi_n\|\sggs}$$ converges to $0$ in $\E\sggs^\times$ and
$\|X_\lambda^{-1}\tau_n\|\sggs =1$, for every $n \in {\mb N}$. Put $\varphi_n=X_\lambda^{-1}\tau_n$, for every $n \in {\mb N}$. Then, $\ggs_\lambda^{\varphi_n}= X_\lambda \varphi_n =\tau_n =\frac{\ggs_\lambda^{\xi_n}}{\|\xi_n\|\sggs}\to 0$, as $n\to \infty$.

Since $\|\varphi_n\|\sggs =1$, for every $n \in {\mb N}$, we get
\begin{equation}\label{eqn_gi}|\ggs(\varphi_n, \varphi_n)-\lambda\ip{\varphi_n}{\varphi_n}|\leq \sup_{\|\eta\|\sggs\leq 1}\left|\ip{\ggs_\lambda^{\varphi_n}}{\eta}\right|=\|\ggs_\lambda^{\varphi_n}\|\sggs^\times\to 0, \; \mbox{as }n\to \infty.\end{equation}

Let us put $\psi_n= \varphi_n/\|\varphi_n\|$. Then,
$$ \|\varphi_n\|^2 |\ggs(\psi_n,\psi_n)-\lambda|\to 0 \; \mbox{as }n\to \infty.$$
The condition {\sf{(qc)}} implies that $\inf_{n\in {\mb N}}\|\varphi_n\|>0$. Indeed, were $\inf_{n\in {\mb N}}\|\varphi_n\|=0$, then, since $\|\varphi_n\|>0$ for every $n \in {\mb N}$, there would be a subsequence $\{\varphi_{n_k}\}$ converging to $0$. Then, by \eqref{eqn_gi}, $\lim_{k\to\infty}|\ggs(\varphi_{n_k}, \varphi_{n_k})|=0$.
Thus, by   {\sf{(qc)}} $\|\varphi_{n_k}\|\sggs \to 0$, a contradiction. This, in turn, implies that
  $|\ggs(\psi_n,\psi_n)-\lambda|\to 0, \; \mbox{as }n\to \infty$.  Hence, $\lambda \in \overline{\gn\sggs}$.

Finally, if the range $\ran (X_\lambda)$ is not dense, by the reflexivity of the Banach space $\E\sggs$, there exists $\eta \in \E\sggs$, such that $\ip{X_\lambda \xi}{\eta}=0$, for every $\xi \in \E\sggs$. Clearly we may suppose  $\|\eta\|=1$. Then, we have
$$ \ggs(\eta,\eta)-\lambda = \ip{X_\lambda \eta}{\eta}=0.$$ Thus $\lambda \in \gn\sggs$.
\end{proof}
\beex The condition {\sf{(qc)}} is obviously satisfied by a closed sectorial form, with $\|\cdot\|\sggs= (\Re \Omega (\cdot, \cdot) + (1-\delta) \|\cdot\|^2)^{1/2}$, where $\delta$ is the lower bound of $\Re\Omega.$
\enex

{
\section{Examples}
We collect in this section some examples illustrating the ideas developed in this paper.

\beex Let $\omega$ denote the space of all complex sequences and $\omega_F$ the subspace of $\omega$ consisting of the sequences with a finite number of nonzero components. For $\{a_n\}, \{b_n\}\in \omega_F$,
we take as $\Theta$  the restriction  to $\omega_F \times \omega_F$ of the usual inner product $\iota$ of $\ell^2$ and
$$\Omega(\{a_n\}, \{b_n\})=\sum_{n=1}^\infty \lambda_n a_n \overline{b}_n,$$
with $\{\lambda_n\}\in \omega$.
Then $\Omega$ is a sesquilinear form on $\omega_F \times \omega_F$. If we put
$$\Psi(\{a_n\}, \{b_n\})= \sum_{n=1}^\infty {|\lambda_n|} a_n \overline{b}_n,\quad \{a_n\}, \{b_n\}\in \omega_F$$
then,
$$|\Omega(\{a_n\}, \{a_n\})| \leq \Psi(\{a_n\}, \{a_n\}) , \quad \forall \{a_n\} \in \omega_F.$$
Hence $\M(\Omega)$ is nonempty.

The form $\Psi$ is closable, since it is the restriction to $\omega_F \times \omega_F$ of the closed form
$$\widetilde{\Psi}(\{a_n\}, \{b_n\})= \sum_{n=1}^\infty {|\lambda_n|} a_n \overline{b}_n,$$
defined on the domain
$$ D(\widetilde{\Psi})=\left\{ \{a_n\} \in \ell^2:\, \sum_{n=1}^\infty {|\lambda_n|} |a_n|^2<\infty \right\} .$$

It follows that $\Psi$ is $\Theta$-absolutely continuous  and $\Omega$ is $\Theta$-regular.
Let $H$ be operator defined on $D(H):=D(\widetilde{\Psi})$ by $H\{a_n\}= \{\sqrt{|\lambda_n|}\, a_n\}$
and $Y$ the operator defined on $\omega_F$ by $Y\{a_n\}= \{e^{i\phi_n}a_n\}$ where $\phi_n \in {\rm arg}{\lambda_n}$. 
Then it is immediate to see that
$$ \Omega(\{a_n\}, \{b_n\})= \ip{HY\{a_n\}}{H\{b_n\}}.$$

But is is also evident that $\Omega$ can also be represented as
$$ \Omega(\{a_n\}, \{b_n\})= \ip{T\{a_n\}}{\{b_n\}},$$
where $T$ is the closed operator defined as follows

 $$\left\{\begin{array}{l}D(T)=\left\{\{a_n\}\in \ell^2:\, \sum_{n=1}^\infty |\lambda_n|^2 |a_n|^2 <\infty\right\}\\ T\{a_n\}=\{\lambda_na_n\} \end{array}.\right. $$

 If, for instance,  $\lambda_n = ne^{in}$, the corresponding form $\Omega$ is neither bounded nor sectorial.

 Now we show that the sesquilinear form $\widetilde\Omega$, with domain $D(\widetilde{\Psi})$, defined by
 $$\widetilde{\Omega}(\{a_n\}, \{b_n\})=\sum_{n=1}^\infty \lambda_n a_n \overline{b}_n,\quad \{a_n\}, \{b_n\}\in D(\widetilde{\Psi})$$
 satisfies, for $\Upsilon=-\lambda\iota$, with $\lambda \not\in \overline{\{\lambda_n; n\in {\mb N}\}}$, the condition (ii) of Definition \ref{defn_39}, with the choice $\| \cdot \|_{\sss\Omega} = (\| \cdot \|_2^2 + \widetilde{\Psi}(\cdot,\cdot))^{1/2}$, where $\| \cdot \|_2$ denotes the $\ell^2$-norm. The proof of (i) is in fact very simple.

 Since $\widetilde{\Psi}$ is closed, $D(\widetilde{\Psi})[\| \cdot \|_{\sss\Omega}]$ is a Hilbert space. If $\Lambda \in D(\widetilde{\Psi})^\times$, then by Riesz's lemma, there exists sequence $\{a_n\}\in D(\widetilde{\Psi})$ such that
 $$ \ip{\Lambda}{\{b_n\}} = \sum_{n=1}^\infty a_n\overline{b}_n + \sum_{n=1}^\infty |\lambda_n| \,a_n\overline{b}_n, \quad \forall \{b_n\}\in D(\widetilde{\Psi}).$$
 Let us consider the sequence $\{c_n\}$, with $c_n= \frac{1+|\lambda_n|}{\lambda_n -\lambda}\,a_n$. Then $\{c_n\}\in D(\widetilde{\Psi})$ and
 $$ (\Omega-\lambda\iota) (\{c_n\}, \{b_n\})=  \sum_{n=1}^\infty (1+|\lambda_n|)a_n\overline{b}_n =\ip{\Lambda}{\{b_n\}}, \quad \forall \{b_n\}\in D(\widetilde{\Psi}).$$
\enex
\beex Let $X$ be a set, ${\mc M}$ a $\sigma$-algebra of subsets of $X$, and  $\theta$ a positive measure on ${\mc M}$.
 We denote by $\D$ the linear span of the characteristic functions of $\theta$-measurable subsets of $X$ and define
  $$\Theta (f,g)= \int_X f(x)\overline{g(x)}d\theta(x), \quad f,g \in \D.$$

Let us now consider a complex measure $\omega$ on  ${\mc M}$. Then, as is known \cite[Theorem 6.4]{rudin}, $\omega$ is a finite measure on $X$ and its total variation $|\omega|$ is a positive measure.
We define a sesquilinear form $\Omega$ on $\D$ by
$$ \Omega (f,g)=\int_X f(x)\overline{g(x)}d\omega(x), \quad f,g \in \D.$$
One can easily prove that, for every $f,g \in \D$,
$$ \left| \int_X f(x)\overline{g(x)}d\omega(x)\right|\leq \left(\int_X |f(x)|d|\omega|(x) \right)^{1/2} \left(\int_X |g(x)|d|\omega|(x) \right)^{1/2}.$$
If $|\omega|$ is absolutely continuous with respect to $\theta$, then the sesquilinear form $|\Omega|$ defined on $\D \times \D$ is $\Theta$-absolutely continuous \cite[Lemma 5.1]{hassi} and, therefore, $\Omega$ is $\Theta$-regular. As a consequence of the Radon-Nikodym theorem for measures one has $d\omega(x) = e^{i \phi(x)}k(x)d\theta$, with $\phi$ a real-valued measurable function and $k\in L^1(\theta), k\geq 0$. We finally get the representation
$$ \Omega(f,g)= \int_X f(x)\overline{g(x)}e^{i \phi(x)}k(x)d\theta, \quad f,g \in \D.$$
Identifying $H$ with the multiplication operator by $\sqrt{k(x)}$ and $Y$ with the multiplication operator by $e^{i \phi(x)}$,
we get, according to Theorem \ref{thm_positivedom}, the representation
$$\Omega(f,g) =\ip{HYf}{Hg}_{\sss \theta}, \quad f,g \in \D,$$
the inner product on the right hand side being that of $L^2(\theta)$. \enex

\beex Let $S$, $T$ be closable linear operators in Hilbert space $\H$. We suppose that $\D:=D(S)\cap D(T)$ is dense in $\H$.
We will show that the sesquilinear form $\Omega$ on $\D\times\D$ defined by
$$\Omega(\xi,\eta)= \ip{S\xi}{T\eta}, \quad \xi, \eta \in \D$$
is $\iota$-regular (where $\iota$ denotes the inner product of $\H)$.
Let us consider the positive selfadjoint operator $H=(I+S^*\overline{S}+ T^*\overline{T})^{1/2}$ whose domain contains $\D$. Then it is easy to see that the positive sesquilinear $\Psi$ defined by
$\Psi(\xi,\eta)=\ip{H\xi}{H\eta}$, $\xi, \eta \in \D$, is a member of $\M(\Omega)$ and that $\Psi$ is $\iota$-absolutely continuous.  Thus $\Omega$ is $\iota$-regular and can be represented as
$$\Omega(f,g) =\ip{HYf}{Hg}, \quad f,g \in \D.$$
In most cases, the form of $Y$ remains implicit.
\enex
}


\noindent{\bf Acknowledgement.} This research was supported by the Gruppo Nazionale per l'Analisi Matematica e le sue Applicazioni (GNAMPA) of INdAM.
We thank Mr. R. Corso for pointing out some inaccuracies in a previous version of this paper.

\end{document}